\documentclass[11pt,reqno]{amsart}
\usepackage{amsmath,amssymb}

 \makeatletter
 \evensidemargin  \oddsidemargin
 \makeatother

\begin{document}
\thispagestyle{empty}
 \setcounter{page}{1}

\begin{center}
{\large\bf Stability of a mixed type quadratic, cubic  and quartic
 functional equation

\vskip.20in

{\bf  M. Eshaghi Gordji } \\[2mm]

{\footnotesize Department of Mathematics,
Semnan University,\\ P. O. Box 35195-363, Semnan, Iran\\
[-1mm] e-mail: {\tt madjid.eshaghi@gmail.com}}

{\bf  S. Kaboli-Gharetapeh } \\[2mm]

{\footnotesize Department of Mathematics,
Payame Nour  University of Mashhad,\\  Mashhad, Iran\\
[-1mm] e-mail: {\tt simin.kaboli@gmail.com}}

{\bf  S. Zolfaghari } \\[2mm]

{\footnotesize Department of Mathematics,
Semnan University,\\ P. O. Box 35195-363, Semnan, Iran\\
[-1mm] e-mail: {\tt zolfagharys@Yahoo.com}}}
\end{center}
\vskip 5mm
 \noindent{\footnotesize{\bf Abstract.} In this paper, we obtain the general solution and the generalized
  Hyers-Ulam Rassias stability of the functional equation

$$3(f(x+2y)+f(x-2y))=12(f(x+y)+f(x-y))+4f(3y)-18f(2y)+36f(y)-18f(x).$$

 \vskip.10in

  \newtheorem{df}{Definition}[section]
  \newtheorem{rk}[df]{Remark}
   \newtheorem{lem}[df]{Lemma}
   \newtheorem{thm}[df]{Theorem}
   \newtheorem{pro}[df]{Proposition}
   \newtheorem{cor}[df]{Corollary}
   \newtheorem{ex}[df]{Example}

 \setcounter{section}{0}
 \numberwithin{equation}{section}

\vskip .2in

\begin{center}
\section{Introduction}
\end{center}
The stability problem of functional equations originated from a
question of Ulam [17] in 1940, concerning the stability of group
homomorphisms. Let $(G_1,.)$ be a group and let $(G_2,*)$ be a
metric group with the metric $d(.,.).$ Given $\epsilon >0$, dose
there exist a $\delta
>0$, such that if a mapping $h:G_1\longrightarrow G_2$ satisfies the
inequality $d(h(x.y),h(x)*h(y)) <\delta$ for all $x,y\in G_1$, then
there exists a homomorphism $H:G_1\longrightarrow G_2$ with
$d(h(x),H(x))<\epsilon$ for all $x\in G_1?$ In the other words,
Under what condition dose there exists a homomorphism near an
approximate homomorphism? The concept of stability for functional
equation arises when we replace the functional equation by an
inequality which acts as a perturbation of the equation. In 1941, D.
H. Hyers [8] gave a first affirmative  answer to the question of
Ulam for Banach spaces. Let $f:{E}\longrightarrow{E'}$ be a mapping
between Banach spaces such that
$$\|f(x+y)-f(x)-f(y)\|\leq \delta $$
for all $x,y\in E,$ and for some $\delta>0.$ Then there exists a
unique additive mapping $T:{E}\longrightarrow{E'}$ such that
$$\|f(x)-T(x)\|\leq \delta$$
for all $x\in E.$ Moreover if $f(tx)$ is continuous in t for each
fixed $x\in E,$ then $T$ is linear. In 1978, Th. M. Rassias [15]
provided a generalization of Hyers' Theorem which allows the Cauchy
difference to be unbounded. In 1991, Z. Gajda [4] answered the
question for the case $p>1$, which was rased by Rassias.  This new
concept is known as Hyers-Ulam-Rassias stability of functional
equations (see [1,2], [4-9], [13-14]).

The functional equation $$f(x+y)+f(x-y)=2f(x)+2f(y).\eqno \hspace
{0.5
 cm}(1.1)$$
 is related to symmetric bi-additive function. It is natural
 that this equation is called a quadratic functional equation.
 In particular, every solution of the quadratic equation (1.1) is
 said to be a quadratic function. It is well known that a function
 $f$ between real vector spaces is quadratic if and only if there
 exits a unique symmetric bi-additive function $B$ such that
 $f(x)=B(x,x)$ for all $x$ (see [1,11]).
The bi-additive function $B$ is given by
$$B(x,y)=\frac{1}{4}(f(x+y)-f(x-y)).\eqno \hspace {0.5
 cm}(1.2)$$
 Hyers-Ulam-Rassias stability problem for the quadratic functional
 equation
 (1.1) was proved by Skof for functions
 $f:A\longrightarrow B$, where A is normed space and B Banach
 space (see [16]).
 Cholewa [2] noticed that the Theorem of Skof is still true if
 relevant domain $A$ is replaced an abelian group. In the paper [3]
 , Czerwik proved the Hyers-Ulam-Rassias stability of the equation
 (1.1). Grabiec [6] has generalized these result mentioned above.

Jun and Kim [10] introduced the following functional equation
 $$f(2x+y)+f(2x-y)=2f(x+y)+2f(x-y)+12f(x) \eqno(1.3)$$
and they established the general solution and the generalized
Hyers-Ulam-Rassias stability for the  functional equation (1.3). The
$f(x)=x^3$ satisfies the functional equation (1.3), which is called
a cubic functional equation. Every solution of the cubic functional
equation is said to be a cubic function. Jun and Kim proved that  a
function
 $f$ between real vector spaces X and Y is a solution of (1.3) if and only if there
 exits a unique function $C:X\times X\times X\longrightarrow Y$ such that
 $f(x)=C(x,x,x)$ for all $x\in X,$ and $C$ is symmetric for each
 fixed one variable and is additive for fixed two variables.

In [12], Won-Gil Prak and Jea Hyeong Bae, considered the following
quartic functional equation:
$$f(2x+y)+f(2x-y)=4(f(x+y)+f(x-y))+24f(x)-6f(y).\hspace {2.9cm}(1.4)$$
In fact they proved that a function
 f between real vector spaces X and Y is a solution of (1.3) if and only if there
 exits a unique symmetric multi-additive function $B:X\times X\times X\times X\longrightarrow Y$ such that
 $f(x)=B(x,x,x,x)$ for all $x$. It is easy to show that
 the function $f(x)=x^4$ satisfies the functional equation (1.4), which is called a
 quartic functional equation and every solution of the quartic functional equation is said to be a
 quartic function.

We deal with the next functional equation deriving from quadratic,
cubic and quartic
 functions:

$$3(f(x+2y)+f(x-2y))=12(f(x+y)+f(x-y))
+4f(3y)-18f(2y)+36f(y)-18f(x). \hspace {0.5cm}(1.5)$$

It is easy to see that
 the function $f(x)=ax^2+bx^3+cx^4$ is a solution of the functional equation (1.5). In the
 present paper we investigate the general solution and the generalized
  Hyers-Ulam-Rassias stability of the functional equation (1.5).

\vskip 5mm
\section{ General solution}

In this section we establish the general solution of functional
equation (1.5).

\begin{thm}\label{t2}
Let $X$,$Y$ be vector spaces,  and let  $f:X\longrightarrow Y$  be a
function. Then $f$ satisfies (1.5) if and only if there exist a
unique  symmetric function $Q_1:X\times X \longrightarrow Y,$ a
unique function $C:X\times X\times X \longrightarrow Y$ and a unique
symmetric  multi-additive  function $Q_2:X\times X\times X\times X
\rightarrow Y$
 such that
 $f(x)=Q_1(x,x)+C(x,x,x)+Q_2(x,x,x,x)$ for all $x\in X$, and that $Q_1$ is additive for each fixed one variable,
 $C$ is symmetric for each
 fixed one variable and is additive for fixed two variables.
\end{thm}
\begin{proof} Suppose there exist a
 symmetric function $Q_1:X\times X \longrightarrow Y,$ a
 function $C:X\times X\times X \longrightarrow Y$ and a
symmetric multi-additive function $Q_2:X\times X\times X\times X
\rightarrow Y$
 such that
 $f(x)=Q_1(x,x)+C(x,x,x)+Q_2(x,x,x,x)$ for all $x\in X$, and that $Q_1$ is additive for each fixed one variable,
 $C$ is symmetric for each
 fixed one variable and is additive for fixed two variables. Then it is easy to see
 that $f$ satisfies (1.5). For the converse let  $f$ satisfies (1.5). We decompose $f$ into the even part and odd
part by setting

$$f_e(x)=\frac{1}{2}(f(x)+f(-x)),~~\hspace {0.3 cm}f_o(x)=\frac{1}{2}(f(x)-f(-x)),$$
for all $x\in X.$ By (1.5), we have
\begin{align*}
3(f_e(x+2y)&+f_e(x-2y))=\frac{1}{2}[3f(x+2y)+3f(-x-2y)+3f(x-2y)+3f(-x+2y)]\\
&=\frac{1}{2}[3f(x+2y)+3f(x-2y)]+\frac{1}{2}[3f((-x)+(-2y))+3f((-x)-(-2y))]\\
&=\frac{1}{2}[12f(x+y)+12f(x-y)+4f(3y)-18f(2y)+36f(y)-18f(x)]\\
&+\frac{1}{2}[12f(-x-y)+12f(-x+y)+4f(-3y)-18f(-2y)+36f(-y)-18f(-x)]\\
&=12[\frac{1}{2}(f(x+y)+f(-(x+y)))]+12[\frac{1}{2}(f(x-y)+f(-(x-y)))]\\
&+4[\frac{1}{2}(f(3y)+f(-3y))]-18[\frac{1}{2}(f(2y)+f(-2y))]\\
&+36[\frac{1}{2}(f(y)+f(-y))]-18[\frac{1}{2}(f(x)+f(-x))]\\
&=12(f_e(x+y)+f_e(x-y))+4f_e(3y)-18f_e(2y)+36f_e(y)-18f_e(x)\\
\end{align*}
for all $x,y\in X.$ This means that $f_e$ satisfies (1.5), or

\begin{align*}
3(f_e(x+2y)+f_e(x-2y))&=12(f_e(x+y)+f_e(x-y))\\
&+4f_e(3y)-18f_e(2y)+36f_e(y)-18f_e(x). \hspace {1.9cm}(1.5(e))
\end{align*}
Now we show that the mapping $g:X \rightarrow Y$ defined by
$g(x):=f_e(2x)-16f_e(x)$ is quadratic and the mapping
$h:X\rightarrow Y$ defined by $h(x):=f_e(2x)-4f_e(x)$ is quartic.
 putting $x=y=0$ in
(1.5(e)), we get $f_e(0)=0$. Setting $x=0$ in (1.5(e)), by
evenness of $f_e$ we obtain

$$f_e(3y)=6f_e(2y)-15f_e(y). \eqno \hspace {4.5cm}(2.1)$$

Hence, according to (2.1),  (1.5(e)) can be written as

$$f_e(x+2y)+f_e(x-2y)=4f_e(x+y)+4f_e(x-y)-8f_e(y)+2f_e(2y)-6f_e(x).\eqno \hspace {1.5cm}(2.2)$$
Interchanging $x$ with $y$ in (2.2) gives the equation

$$f_e(2x+y)+f_e(2x-y)=4f_e(x+y)+4f_e(x-y)-8f_e(x)+2f_e(2x)-6f_e(y). \eqno \hspace {2cm}(2.3)$$

With the substitution  $y:=x+y$ in (2.3), we have

$$f_e(3x+y)+f_e(x-y)=4f_e(2x+y)-6f_e(x+y)+4f_e(y)+2f_e(2x)-8f_e(x). \eqno \hspace {2cm}(2.4)$$
Replacing $y$ by $-y$ in (2.4), gives

$$f_e(3x-y)+f_e(x+y)=4f_e(2x-y)-6f_e(x-y)+4f_e(y)+2f_e(2x)-8f_e(x). \eqno \hspace {2cm}(2.5)$$
If we add (2.4) to (2.5), we have

\begin{align*}
f_e(3x+y)+f_e(3x-y)&=4f_e(2x+y)+4f_e(2x-y)\\
&-7f_e(x+y)-7f_e(x-y)+8f_e(y)+4f_e(2x)-16f_e(x). \hspace {1cm}(2.6)
\end{align*}

Setting $x+y$ instead of $x$ in (2.3), we get

\begin{align*}
f_e(2x+3y)+f_e(2x+y)&=4f_e(x+2y)-8f_e(x+y)\\
&+2f_e(2(x+y))-6f_e(y)+4f_e(x).  \hspace {3.9cm}(2.7)
\end{align*}

Which on substitution of $-y$ for $y$ in (2.7) gives

\begin{align*}
f_e(2x-3y)+f_e(2x-y)&=4f_e(x-2y)-8f_e(x-y)\\
&+2f_e(2(x-y))-6f_e(y)+4f_e(x).  \hspace {3.9cm}(2.8)
\end{align*}

By adding (2.7) and (2.8), we lead to

\begin{align*}
f_e(2x+3y)+f_e(2x-3y)&=4f_e(x+2y)+4f_e(x-2y)-f_e(2x+y)-f_e(2x-y)\\
&+2f_e(2(x+y))+2f_e(2(x-y))-8f_e(x+y)\\
&-8f_e(x-y)-12f_e(y)+8f_e(x). \hspace {3.9cm}(2.9)
\end{align*}

Putting $y:=2y$ in (2.6) to obtain

\begin{align*}
f_e(3x+2y)+f_e(3x-2y)&=4f_e(2(x+y))+4f_e(2(x-y))\\
&-7f_e(x+2y)-7f_e(x-2y)\\
&+8f_e(2y)+4f_e(2x)-16f_e(x). \hspace {4cm}(2.10)
\end{align*}

Interchanging $x$ and $y$ in (2.9) to get

\begin{align*}
f_e(3x+2y)+f_e(3x-2y)&=4f_e(2x+y)+4f_e(2x-y)-f_e(x+2y)-f_e(x-2y)\\
&+2f_e(2(x+y))+2f_e(2(x-y))-8f_e(x+y)\\
&-8f_e(x-y)-12f_e(x)+8f_e(y). \hspace {3.9cm}(2.11)
\end{align*}

If we compare (2.10) and  (2.11) and utilizing  (2.2) and (2.3),
we conclude that

\begin{align*}
&[f_e(2(x+y))-16f_e(x+y)]+[f_e(2(x-y))-16f_e(x-y)]\\
&=2[f_e(2x)-16f_e(x)]+2[f_e(2y)-16f_e(y)]
\end{align*}

for all $x,y\in X.$ The last equality means that

$$g(x+y)+g(x-y)=2g(x)+2g(y)$$

for all $x,y\in X.$ Therefore the mapping $g:X \rightarrow Y$ is
quadratic.

With the substitutions $x:=2x$ and $y:=2y$ in (2.3), we have

\begin{align*}
f_e(2(2x+y))+f_e(2(2x-y))&=4f_e(2(x+y))+4f_e(2(x-y))\\
&-6f_e(2y)+2f_e(4x)-8f_e(2x). \hspace {3.5cm}(2.12)
\end{align*}

Let $g:X\rightarrow Y$ be the quadratic mapping defined above.
Since $g(2x)=4g(x)$

for all $x\in X,$ then

$$f_e(4x)=20f_e(2x)-64f_e(x) \eqno \hspace {1.5cm} (2.13)$$

for all $x,y\in X.$

Hence, according to (2.13),  (2.12) can be written as

\begin{align*}
f_e(2(2x+y))+f_e(2(2x-y))&=4f_e(2(x+y))+4f_e(2(x-y))\\
&-6f_e(2y)+32f_e(2x)-128f_e(x). \hspace {3.5cm}(2.14)
\end{align*}

Interchanging $x$ with $y$ in (2.14) gives the equation

\begin{align*}
f_e(2(x+2y))+f_e(2(x-2y))&=4f_e(2(x+y))+4f_e(2(x-y))\\
&-6f_e(2x)+32f_e(2y)-128f_e(y). \hspace {3.5cm}(2.15)
\end{align*}

By multiplying by 4 in (2.2) and subtract the last equation from
(2.15), we  arrive at

\begin{align*}
h(x+2y)+h(x-2y)&=[f_e(2(x+2y))-4f_e(x+2y)]+[f_e(2(x-2y))-4f_e(x-2y)]\\
&=4[f_e(2(x+y))-4f_e(x+y)]+4[f_e(2(x-y))-4f_e(x-y)]\\
&+24[f_e(2y)-4f_e(y)]-6[f_e(2x)-4f_e(x)]\\
&=4h(x+y)+4h(x-y)+24h(y)-6h(x)
\end{align*}

for all $x,y\in X.$ Therefore the mapping $h:X\rightarrow Y$ is
quartic. On the other hand we have
$f_e(x)=\frac{1}{12}h(x)-\frac{1}{12}g(x)$ for all $x\in X.$ This
means that $f_e$ is quartic-quadratic function. Then there exist a
unique symmetric function $Q_1:X\times X\longrightarrow Y$ and a
unique symmetric multi-additive function $Q_2:X\times X\times
X\times X \rightarrow Y$ such that
 $f_e(x)=Q_1(x,x)+Q_2(x,x,x,x)$ for all $x\in X,$ and $Q_1$ is additive for each fixed one
 variable.

On the other hand we can show that $f_o$ satisfies (1.5), or

\begin{align*}
3(f_o(x+2y)+f_o(x-2y))&=12(f_o(x+y)+f_o(x-y))\\
&+4f_o(3y)-18f_o(2y)+36f_o(y)-18f_o(x). \hspace {1.9cm}(1.5(o))
\end{align*}

Setting $x=y=0$ in (1.5(o)) to obtain $f_o(0)=0.$ Putting $x=0$ in
(1.5(o)), then by oddness of $f_o,$  we have

$$2f_o(3y)=9f_o(2y)-18f_o(y). \eqno \hspace {2.5cm} (2.16)$$
Hence (1.5(o)) can be written as

$$f_o(x+2y)+f_o(x-2y)=4f_o(x+y)+4f_o(x-y)-6f_o(x). \eqno \hspace {1.5cm}(2.17)$$
 Replacing  $x$ by $y$ in (1.5(o)) to get

$$f_o(3y)=6f_o(2y)-21f_o(y). \eqno \hspace {2.5cm} (2.18)$$
 By comparing
(2.16) with (2.18), we arrive at

$$f_o(2y)=8f_o(y). \eqno \hspace {5.5cm}(2.19)$$
From the substitution $x:=2x$ in (2.17) and (2.19), it follows
that

$$f_o(2x+y)+f_o(2x-y)=2f_o(x+y)+2f_o(x-y)+12f_o(x).$$
This shows that $f_o$ is cubic. Thus there exists  a unique function
$C:X\times X\times X \longrightarrow Y$ such that $f_o(x)=C(x,x,x)$
for all $x\in X,$ and  $C$ is symmetric for each
 fixed one variable and is additive for fixed two variables.
  Thus for all $x\in X$, we have
$$f(x)=f_e(x)+f_o(x)=Q_1(x,x)+Q_2(x,x,x,x)+C(x,x,x).$$
 This completes the proof of Theorem.
\end{proof}

The following Corollary is an alternative result of above Theorem.

\begin{cor}\label{c1} Let $X$,$Y$
be vector spaces,  and let  $f:X\longrightarrow Y$  be a function
satisfies (1.5). Then the following assertions hold.

a) If f is even function, then f is quartic-quadratic.

b) If f is odd function, then f is cubic.
\end{cor}

\section{ Stability  }

We now investigate the generalized Hyers-Ulam-Rassias stability
problem for functional equation  (1.5).  From now on, let X be a
real vector space and let Y be a Banach space. Now before taking up
the main subject, given $f:X\rightarrow Y$, we define the difference
operator $D_f:X\times X \rightarrow Y$ by

$$D_{f}(x,y)=3[f(x+2y)+f(x-2y)]-12[f(x+y)+f(x-y)]-4f(3y)+18f(2y)-36f(y)+18f(x)$$
for all $x,y \in X.$  We consider the following functional
inequality:

$$\|D_f(x,y)\|\leq\phi(x,y)$$
for an upper bound $\phi:X \times X\rightarrow [0,\infty).$

\begin{thm}\label{t'2} Let $s\in\{1,-1\}$ be fixed. Suppose that an odd mapping $f:X\rightarrow Y$ satisfies

 $$\| D_f(x,y)\|\leq\phi(x,y)\eqno \hspace {0.5cm} (3.1)$$
for all $x,y\in X.$ If the upper bound $\phi:X\times
X\rightarrow[0,\infty)$ is a mapping such that

$$\sum^{\infty}_{i=1} 8^{si}
[\phi(2^{-si}x,2^{-si}y)+4\phi(0,2^{-si}x)]<\infty ,$$ and that
$\lim_n 8^{sn} \phi(2^{-sn}x,2^{-sn}y)=0$ for all $x,y\in X.$ Then
the limit $C(x):=\lim_n 8^{sn} f(2^{-sn}x)$ exists for all $x\in X,$
and $C:X\rightarrow Y$ is a unique cubic function satisfies (1.5),
and

$$\|f(x)-C(x)\|\leq \frac{1}{6}\sum^{\infty}_{i=\frac{|s+1|}{2}} 8^{si-1} \phi(0,2^{-si}x)
+\frac{4}{6}\sum^{\infty}_{i=\frac{|s+1|}{2}} 8^{si-1}
\phi(2^{-si}x,2^{-si}x),\eqno \hspace {0.5cm}(3.2)$$ for all $x\in
X.$
\end{thm}
\begin{proof} Putting $x=0$ in (3.1) to get

$$\parallel 4f(3y)-18f(2y)+36f(y)\parallel\leq \phi(0,y). \eqno \hspace {2cm}
(3.3)$$ Now replacing $y$ by $x$ in (3.1) to obtain

$$\parallel f(3y)-6f(2y)+21f(y)\parallel\leq \phi(y,y). \eqno \hspace {2cm}
(3.4)$$ combining (3.3) with (3.4) yields

$$\parallel \frac{f(2y)}{8}-f(y)\parallel\leq
\frac{1}{6\times8}\phi(0,y)+\frac{4}{6\times8}\phi(y,y). \eqno
\hspace {2cm}(3.5)$$  From the inequality (3.5) we use iterative
methods and induction on $n$ to prove our next relation.

$$\parallel \frac{f(2^nx)}{8^n}-f(x)\parallel \leq \frac{1}{6}\sum^{n-1}_{i=0}
\frac{\phi(0,2^ix)}{8^{i+1}}+\frac{4}{6}\sum^{n-1}_{i=0}\frac{\phi(2^ix,2^ix)}{8^{i+1}}.
\eqno \hspace {0.5cm}(3.6)$$

Dividing (3.6) by $8^m,$ and then replacing $x$ by $2^mx,$  it
follows that

\begin{align*}
\parallel \frac{f(2^{m+n}x)}{8^{m+n}}-\frac{f(2^mx)}{8^m}\parallel
&\leq\frac{1}{6}\sum^{n-1}_{i=0}
\frac{\phi(0,2^{m+i}x)}{8^{m+i+1}}+\frac{4}{6}\sum^{n-1}_{i=0} \frac{\phi(2^{m+i}x,2^{m+i}x)}{8^{m+i+1}}\\
&=\frac{1}{6}\sum^{m+n-1}_{i=m}
\frac{\phi(0,2^ix)}{8^{i+1}}+\frac{4}{6}\sum^{m+n-1}_{i=m}\frac{\phi(2^ix,2^ix)}{8^{i+1}}.\hspace
{1.8cm} (3.7)
\end{align*}

This shows that $\{\frac{f(2^nx)}{8^n}\}$ is a Cauchy sequence in
Y, by taking the limit $m\rightarrow\infty$ in (3.7). Since Y is a
Banach space, it follows that the sequence
$\{\frac{f(2^nx)}{8^n}\}$ converges. Now we define $C:X\rightarrow
Y$ by $C(x):=\lim_n \frac{f(2^nx)}{8^n}$ for all $x\in X.$
Obviously  (3.2) holds for $s=-1.$ It is easy to see that
$C(-x)=-C(x)$ for all $x\in X.$ By using (3.1) we have

$$\parallel D_C(x,y)\parallel=\lim_n \frac{1}{8^n} \parallel
D_f(2^nx,2^ny)\parallel\leq\lim_n \frac{1}{8^n} \phi(2^nx,2^ny)=0$$
for all $x,y\in X.$ Hence by Corollary 2.2,  C is cubic. It remains
to show that C is unique. Suppose that there exists a cubic function
$C':X\rightarrow Y$ which satisfies (1.5) and (3.2). Since
$C(2^nx)=8^n C(x),$ and $C'(2^nx)=8^n C'(x),$ for all $x\in X,$ we
have

\begin{align*}
\parallel C(x)-C'(x)\parallel
&=\frac{1}{8^n}\parallel C(2^nx)-C'(2^nx)\parallel\\
&\leq \frac{1}{8^n}\parallel C(2^nx)-f(2^nx)\parallel+
\frac{1}{8^n}\parallel
C'(2^nx)-f(2^nx)\parallel\\
&\leq \frac{1}{6}\sum^{\infty}_{i=0}
\frac{1}{8^{n+i}}\phi(0,2^{n+i}x)+\frac{4}{6}\sum^{\infty}_{i=0}\frac{1}{8^{n+i}}\phi(2^{n+i}x,2^{n+i}x)
\end{align*}
for all $x\in X.$ By taking $n\rightarrow\infty$ in this inequality,
it follows that $C(x)=C'(x)$ for all $x\in X.$ Which gives the
conclusion for  $s=-1.$ On the other hand by replacing
 $2y$ by $x$  in (3.5) and  multiplying the result by $8,$ we get

$$\|f(x)-8f(\frac{x}{2})\|\leq \frac{1}{6}\phi(0,\frac{x}{2})+\frac{4}{6} \phi(\frac{x}{2},\frac{x}{2}).
\eqno \hspace {0.5cm}(3.8)$$ From (3.8) we use iterative methods
and induction on n to obtain

 $$\|f(x)-8^nf(\frac{x}{2^n})\|\leq \frac{1}{6}\sum^{n-1}_{i=0}8^i\phi(0,\frac{x}{2^{i+1}})
 +\frac{4}{6}\sum^{n-1}_{i=0}8^i \phi(\frac{x}{2^{i+1}},\frac{x}{2^{i+1}})
\eqno \hspace {0.5cm}(3.9)$$ for all $x\in X.$

Now multiplying  both sides of (3.9) with $8^m$ and replacing $x$ by
$\frac{x}{2^m}$ in (3.9) to get

\begin{align*}
\|f(\frac{x}{2^m})-8^{n+m}f(\frac{x}{2^{n+m}})\|&\leq
\frac{1}{6}\sum^{n-1}_{i=0}8^{m+i}\phi(0,\frac{x}{2^{m+i+1}})
 +\frac{4}{6}\sum^{n-1}_{i=0}8^{m+i}
 \phi(\frac{x}{2^{m+i+1}},\frac{x}{2^{m+i+1}})\\
 &=\frac{1}{6}\sum^{m+n-1}_{i=m}8^i\phi(0,\frac{x}{2^{i+1}})
 +\frac{4}{6}\sum^{m+n-1}_{i=m} 8^i
 \phi(\frac{x}{2^{i+1}},\frac{x}{2^{i+1}}).
 \hspace {0.5cm}(3.10)
\end{align*}

 By taking $m\rightarrow\infty$ in
(3.10), it follows that $\{8^n f(\frac{x}{2^n})\}$ is a Cauchy
sequence in Y. Then $C(x):=\lim_n 8^n f(\frac{x}{2^n})$ exists for
all $x\in X.$ Obviously (3.2) holds for $s=1.$ The rest of proof
is similar to the proof of the case $s=-1.$

\end{proof}

\begin{thm}\label{t2} Suppose  an even function $f:X\rightarrow Y$ satisfies
 $$\|D_f(x,y)\|\leq \phi(x,y)\eqno \hspace {0.5cm}(3.11)$$
 for all $x,y\in X.$  If the upper bound $\phi:X\times X\rightarrow
 [0,\infty)$ is a mapping such that

 $$\sum^{\infty}_{i=1} 4^i
 [\phi(\frac{x}{2^i},\frac{x}{2^{i+1}})+\phi(\frac{x}{2^i},\frac{x}{2^i})]<\infty  \eqno \hspace {3.5cm} (3.12)$$
 for all $x\in X,$ and that $\lim_n 4^n \phi(\frac{x}{2^n},\frac{y}{2^n})=0$ for all $x,y\in
X.$  Then the limit

 $$Q_1(x):=\lim_n 4^n
[f(\frac{x}{2^{n-1}})-16f(\frac{x}{2^n})]$$ exists for all $x\in
X,$ and $Q_1:X\rightarrow Y$ is a unique quadratic function
satisfies (1.5), and
 $$ \|f(2x)-16f(x)-Q_1(x)\|\leq \sum^{\infty}_{i=0} 4^i[\frac{1}{3}\phi(\frac{x}{2^i},\frac{x}{2^{i+1}})
 +\frac{16}{3}\phi(\frac{x}{2^{i+1}},\frac{x}{2^{i+1}})]. \eqno \hspace {3.5cm} (3.13)$$
for all $x\in X.$

\end{thm}

\begin{proof} Replacing $x$ by $2y$ in (3.11) to obtain

 $$ \|3f(4y)-16f(3y)+36f(2y)-48f(y)\|\leq \phi(2y,y). \eqno \hspace {3.5cm} (3.14)$$
 Replacing $x$ by $y$ in (3.11) to get

$$ \|f(3y)-6f(2y)+15f(y)\|\leq \phi(y,y). \eqno \hspace {3.5cm} (3.15)$$
By combining (3.14) and (3.15)  we lead to

\begin{align*}
\|f(4x)-20f(2x)+64f(x)\|&=\|\frac{1}{3}[3f(4y)-16f(3y)+36f(2y)-48f(y)]\\
&+\frac{16}{3}[f(3y)-6f(2y)+15f(y)]\|\\
&\leq\frac{1}{3}
 \phi(2x,x)+\frac{16}{3} \phi(x,x)  \hspace {3.5cm} (3.16)
\end{align*}

  for all $x\in
 X.$ Put $g(x)=f(2x)-16f(x)$ for all $x\in X.$ Then by (3.16) we
 have

$$\|g(2x)-4g(x)\|\leq \frac{1}{3} \phi(2x,x)+\frac{16}{3} \phi(x,x). \eqno \hspace {3.5cm} (3.17)$$
Replacing $x$ by $\frac{x}{2}$ in (3.17) to get

$$\|g(x)-4g(\frac{x}{2})\|\leq \frac{1}{3} \phi(x,\frac{x}{2})+\frac{16}{3} \phi(\frac{x}{2},\frac{x}{2}).
 \eqno \hspace {3.5cm} (3.18)$$ An induction argument now implies that

$$\|g(x)-4^n g(\frac{x}{2^n})\|\leq \sum^{n-1}_{i=0} 4^i [\frac{1}{3}
 \phi(\frac{x}{2^i},\frac{x}{2^{i+1}})+\frac{16}{3} \phi(\frac{x}{2^{i+1}},\frac{x}{2^{i+1}})]
  \hspace {3.5cm} (3.19)$$ for all $x\in X.$  Multiplying both
sides of above inequality by $4^m$ and replacing $x$ by
$\frac{x}{2^m}$ to get

\begin{align*}
\|4^m g(\frac{x}{2^m})-4^{m+n}g(\frac{x}{2^{m+n}})\|&\leq
\sum^{n-1}_{i=0}
4^{i+m}[\frac{1}{3}\phi(\frac{x}{2^{i+m}},\frac{x}{2^{m+i+1}})+\frac{16}{3}
\phi(\frac{x}{2^{m+i+1}},\frac{x}{2^{m+i+1}})]\\
&\leq \sum^{m+n-1}_{i=m}
4^i[\frac{1}{3}\phi(\frac{x}{2^i},\frac{x}{2^{i+1}})+\frac{16}{3}
\phi(\frac{x}{2^{i+1}},\frac{x}{2^{i+1}})].
\end{align*}

Since the right hand side of the above inequality tends to 0 as
$m\rightarrow\infty,$ the sequence $\{4^n g(\frac{x}{2^n})\}$ is
Cauchy. Then the limit $Q_1(x):=\lim_n 4^n g(\frac{x}{2^n})=\lim_n
4^n (f(\frac{x}{2^{n-1}})-16f(\frac{x}{2^n}))$ exists for all
$x\in X.$ On the other hand we have

\begin{align*}
\|Q_1(2x)-4Q_1(x)\|&=\lim_n [4^n
g(\frac{x}{2^{n-1}})-4^{n+1}g(\frac{x}{2^n})]\\
&=4\lim_n[4^{n-1} g(\frac{x}{2^{n-1}})-4^n g(\frac{x}{2^n})]=0
\hspace {3.5cm} (3.20)
\end{align*}
for all $x\in X.$ Let $D_g(x,y):=D_f(2x,2y)-16D_f(x,y)$ for all
$x\in X.$ Then we have

\begin{align*}
D_{Q_1}(x,y)=\lim_n \|4^n
D_g(\frac{x}{2^n},\frac{y}{2^n})\|&=\lim_n 4^n
\|D_f(\frac{x}{2^{n-1}},\frac{y}{2^{n-1}})-16D_f(\frac{x}{2^n},\frac{y}{2^n})\|\\
&\leq\lim_n
4\|4^{n-1}D_f(\frac{x}{2^{n-1}},\frac{y}{2^{n-1}})\|+\lim_n
16\|4^nD_f(\frac{x}{2^n},\frac{y}{2^n})\|\\
&\leq4\lim_n 4^{n-1}
\phi(\frac{x}{2^{n-1}},\frac{y}{2^{n-1}})+16\lim_n 4^n
\phi(\frac{x}{2^n},\frac{y}{2^n})=0
\end{align*}

This means that $Q_1$ satisfies (1.5). Thus by (3.20), it follows
that $Q_1$ is quadratic. It remains to show that $Q_1$ is unique
quadratic function which satisfies (3.13). Suppose that there
exists a quadratic function ${Q'_1}:X\rightarrow Y$  satisfies
(3.13). Since $Q_1(2^nx)=4^n Q_1(x),$ and ${Q'_1}(2^nx)=4^n
{Q'_1}(x)$ for all $x\in X,$ it follows that

\begin{align*}
\parallel Q_1(x)-{Q'_1}(x)\parallel
=4^n\parallel
Q_1(\frac{x}{2^n})-{Q'_1}(\frac{x}{2^n})\parallel&\leq4^n[\parallel
Q_1(\frac{x}{2^n})-f(\frac{2x}{2^n})-16f(\frac{x}{2^n})\parallel\\
&+\parallel
{Q'_1}(\frac{x}{2^n})-f(\frac{2x}{2^n})-16f(\frac{x}{2^n})\parallel]\\
&\leq \sum^{\infty}_{i=n}
4^i[\frac{1}{3}\phi(\frac{x}{2^i},\frac{x}{2^{i+1}})+\frac{16}{3}\phi(\frac{x}{2^{i+1}},\frac{x}{2^{i+1}})]
\end{align*}

for all $x\in X.$ By taking $n\rightarrow\infty$ the right hand side
of above inequality tends to 0. Thus we have $ Q_1(x)={Q'_1}(x)$ for
all $x\in X,$ and the proof of Theorem is complete.
\end{proof}

\begin{thm}\label{t2}  Suppose that an  even function  $f:X\rightarrow Y$
 satisfies

  $$\|D_f(x,y)\|\leq\phi(x,y)\eqno \hspace {0.5cm} (3.21)$$ for  all $x,y \in X.$
If the upper bound $\phi:X\times X\rightarrow [0,\infty)$ is a
mapping such that $$\sum^{\infty}_{i=1} 16^i
[\phi(\frac{x}{2^i},\frac{x}{2^{i+1}})+\phi(\frac{x}{2^i},\frac{x}{2^i})]<\infty
 \eqno \hspace {0.5cm}(3.22)$$ for all $x \in X$ and that $\lim_n 16^n
\phi(\frac{x}{2^n},\frac{y}{2^n})=0$ for  all $x,y \in X,$  then
the limit

$$Q_2(x):=\lim_n 16^n [f(\frac{x}{2^{n-1}})-4f(\frac{x}{2^n})]$$
exists for  all $x\in X,$ and $Q_2:X\rightarrow Y$ is a unique
quartic  function satisfies (1.5) and

$$\|f(2x)-4f(x)-Q_2(x)\|\leq \sum^{\infty}_{i=0} 16^i [\frac{1}{3}\phi(\frac{x}{2^i},\frac{x}{2^{i+1}})
+\frac{16}{3} \phi(\frac{x}{2^{i+1}},\frac{x}{2^{i+1}})],\eqno
\hspace {0.5cm}(3.23)$$ for all $x\in X.$
\end{thm}

\begin{proof}  Similar to the proof of  Theorem 3.2, we can
show that f  satisfies  (3.16). Put $h(x)=f(2x)-4f(x)$ for all
$x\in X.$ Then by (3.16) we have

$$\parallel h(2x)-16h(x)\parallel\leq\frac{1}{3}\phi(2x,x)+\frac{16}{3}\phi(x,x). \eqno
\hspace {0.5cm}(3.24)$$ Replacing $x$ by $\frac{x}{2}$ in (3.24) to
obtain

$$\parallel h(x)-16h(\frac{x}{2})\parallel\leq\frac{1}{3}\phi(x,\frac{x}{2})+
\frac{16}{3}\phi(\frac{x}{2},\frac{x}{2}). \eqno \hspace
{0.5cm}(3.25)$$ By (3.25) we use iterative methods and induction
on n to prove our next relation.

$$\parallel h(x)-16^n h(\frac{x}{2^n})\parallel\leq\sum^{n-1}_{i=0} 16^i[\frac{1}{3}\phi(\frac{x}{2^i},\frac{x}{2^{i+1}})+
\frac{16}{3}\phi(\frac{x}{2^{i+1}},\frac{x}{2^{i+1}})]. \eqno
\hspace {0.5cm}(3.26)$$ Replacing $x$ by $\frac{x}{2^m}$ in (3.26)
and then multiplying the result by $16^m$ to get

\begin{align*}
\parallel 16^m h(\frac{x}{2^m})-16^{m+n} h(\frac{x}{2^{m+n}})\parallel
&\leq\sum^{n-1}_{i=0}
16^{m+i}[\frac{1}{3}\phi(\frac{x}{2^{m+i}},\frac{x}{2^{m+i+1}})+
\frac{16}{3}\phi(\frac{x}{2^{m+i+1}},\frac{x}{2^{m+i+1}})]\\
&= \sum^{m+n-1}_{i=m}
16^i[\frac{1}{3}\phi(\frac{x}{2^i},\frac{x}{2^{i+1}})+\frac{16}{3}\phi(\frac{x}{2^{i+1}},\frac{x}{2^{i+1}})].
\end{align*}

By taking $m\rightarrow \infty$ in above inequality, it follows
that

$$\lim_m \|16^m h(\frac{x}{2^m})-16^{m+n} h(\frac{x}{2^{m+n}})\|=
0.$$ This means that $\{16^n h(\frac{x}{2^n})\}$ is a Cauchy
sequence in Y. Thus the limit $Q_2(x)=\lim_n 16^n
h(\frac{x}{2^n})=\lim_n 16^n
[f(\frac{x}{2^{n-1}})-4f(\frac{x}{2^n})]$ exists for all $x\in X.$
On the other hand we have

\begin{align*}
\|Q_2(2x)-16Q_2(x)\|&=\lim_n \|16^n
h(\frac{x}{2^{n-1}})-16^{n+1}h(\frac{x}{2^n})\|\\
&=16\lim_n \|16^{n-1}h(\frac{x}{2^{n-1}})-16^n h(\frac{x}{2^n})\|=0.
 \hspace {3.5cm}(3.27)
\end{align*}
Set $D_h(x,y)=D_f(2x,2y)-4D_f(x,y)$ for all $x,y\in X.$ Then we have

\begin{align*}
D_{Q_2}(x,y)&=\lim_n \|16^n
D_h(\frac{x}{2^n},\frac{y}{2^n})\|=\lim_n 16^n
\|D_f(\frac{x}{2^{n-1}},\frac{y}{2^{n-1}})-16D_f(\frac{x}{2^n},\frac{y}{2^n})\|\\
&\leq\lim_n 16
\|16^{n-1}D_f(\frac{x}{2^{n-1}},\frac{y}{2^{n-1}})\|+\lim_n 4
\|16^n D_f(\frac{x}{2^n},\frac{y}{2^n})\|\\
&\leq 16\lim_n 16^{n-1}
\phi(\frac{x}{2^{n-1}},\frac{y}{2^{n-1}})+4\lim_n 16^n
\phi(\frac{x}{2^n},\frac{y}{2^n})=0.
\end{align*}
This means that $Q_2$ satisfies (1.5). By (3.27) it follows that
$Q_2$ is quartic function. To prove the uniqueness property of
$Q_2,$ let $Q'_2:X\rightarrow Y$ be a quartic function which
satisfies (1.5) and (3.23). Since $Q_2(2^nx)=16^n Q_2(x),$ and
${Q'_2}(2^nx)=16^n {Q'_2}(x)$ for all $x\in X,$ then

\begin{align*}
\parallel Q_2(x)-{Q'_2}(x)\parallel
=16^n\parallel
Q_2(\frac{x}{2^n})-{Q'_2}(\frac{x}{2^n})\parallel&\leq16^n[\parallel
Q_2(\frac{x}{2^n})-f(\frac{2x}{2^n})-4f(\frac{x}{2^n})\parallel\\
&+\parallel
{Q'_2}(\frac{x}{2^n})-f(\frac{2x}{2^n})-4f(\frac{x}{2^n})\parallel]\\
&\leq 2\sum^{\infty}_{i=n}
16^i[\frac{1}{3}\phi(\frac{x}{2^i},\frac{x}{2^{i+1}})+\frac{16}{3}\phi(\frac{x}{2^{i+1}},\frac{x}{2^{i+1}})]
\end{align*}
for all $x\in X.$ Let $n\rightarrow\infty$ in above inequality. Then
by (3.22), we have $Q_2(x)={Q'_2}(x)$  for all $x\in X.$ This
complete the proof of Theorem.
\end{proof}

\begin{thm}\label{t2} Suppose that an even mapping $f:X\rightarrow
Y$ satisfies $\|D_f(x,y)\|\leq\phi(x,y)$ for all $x,y\in X.$ If
the upper bound $\phi:X\times X\rightarrow [0,\infty)$ satisfies
$$\sum^{\infty}_{i=1} 16^i
\phi(\frac{x}{2^i},\frac{x}{2^{i+1}})+\sum^{\infty}_{i=1} 16^i
\phi(\frac{x}{2^i},\frac{x}{2^i})<\infty, \eqno \hspace {0.5cm}
(3.28)$$ and  $\lim_n 16^n \phi(\frac{x}{2^n},\frac{y}{2^n})=0$ for
all $x,y \in X.$ Then there exist a unique quadratic function
$Q_1:X\rightarrow Y$ and a unique quartic  function
$Q_2:X\rightarrow Y$ such that

$$\parallel f(x)-Q_1(x)-Q_2(x)\parallel\leq \frac{1}{12}\sum^{\infty}_{i=0} (4^i+16^i)
[\frac{1}{3}\phi(\frac{x}{2^i},\frac{x}{2^{i+1}})+\frac{16}{3}
\phi(\frac{x}{2^{i+1}},\frac{x}{2^{i+1}})]\eqno \hspace {0.5cm}
(3.29)$$ for all $x\in X.$
\end{thm}

\begin{proof} By Theorems 3.2 and 3.3, there exist a quadratic
mapping $Q_{o1}:X\rightarrow Y$ and a quartic  mapping
$Q_{o2}:X\rightarrow Y$ such that

$$\|f(2x)-16f(x)-Q_{o1}(x)\|\leq \sum^{\infty}_{i=0} 4^i [\frac{1}{3}\phi(\frac{x}{2^i},\frac{x}{2^{i+1}})+\frac{16}{3}
\phi(\frac{x}{2^{i+1}},\frac{x}{2^{i+1}})] \eqno \hspace
{0.5cm}(3.30)$$ and

$$\|f(2x)-4f(x)-Q_{o2}(x)\|\leq \sum^{\infty}_{i=0} 16^i [\frac{1}{3}\phi(\frac{x}{2^i},\frac{x}{2^{i+1}})
+\frac{16}{3}\phi(\frac{x}{2^{i+1}},\frac{x}{2^{i+1}})] \eqno
\hspace {0.5cm}(3.31)$$ for all $x\in X.$ Combining (3.30) and
(3.31) to obtain

$$\|f(x)+\frac{1}{12}Q_{o1}(x)-\frac{1}{12}Q_{o2}(x)\|\leq\frac{1}{12} [\sum^{\infty}_{i=0}
(4^i+16^i)
\{\frac{1}{3}\phi(\frac{x}{2^i},\frac{x}{2^{i+1}})+\frac{16}{3}\phi(\frac{x}{2^{i+1}},\frac{x}{2^{i+1}})\}].$$
By putting $Q_1(x):=-\frac{1}{12}Q_{o1}(x),$ and
$Q_2(x):=\frac{1}{12}Q_{o2}(x)$ we get (3.29). To prove the
uniqueness  property of $Q_1$ and $Q_2,$ let $Q'_1,Q'_2:X\rightarrow
Y$ be another quadratic and quartic maps satisfying (3.29). Set
$Q''_1=Q_1-Q'_1,$ $Q''_2=Q_2-Q'_2.$ Then by (3.28) we have

\begin{align*}
\lim_n 16^n
\|Q''_1(\frac{x}{2^n})-Q''_2(\frac{x}{2^n})\|&\leq\lim_n 16^n
\|f(\frac{x}{2^n})-Q_1(\frac{x}{2^n})-Q_2(\frac{x}{2^n})\|\\
&+\lim_n 16^n
\|f(\frac{x}{2^n})-Q'_1(\frac{x}{2^n})-Q'_2(\frac{x}{2^n})\|\\
&\leq \frac{2}{12}\sum^{\infty}_{i=0}(16^n\times
(2^i+16^i))[\frac{1}{3}\phi(\frac{x}{2^{n+i}},\frac{x}{2^{n+i+1}})\\
&+\frac{16}{3}\phi(\frac{x}{2^{n+i+1}},\frac{x}{2^{n+i+1}})]\\
&\leq\frac{1}{6}\sum^{\infty}_{i=n}2\times16^i[\frac{1}{3}\phi(\frac{x}{2^i},\frac{x}{2^{i+1}})
+\frac{16}{3}\phi(\frac{x}{2^{i+1}},\frac{x}{2^{i+1}})]=0 \hspace
{0.5cm}(3.32)
\end{align*}
for all $x\in X.$ On the other hand $Q_2$ and $Q'_2$ are quartic,
then $16^nQ''_2(\frac{x}{2^n})=Q''_2(x).$ Thus by (3.32) it follows
that $Q''_2(x)=0$ for all $x\in X.$ It is easy to see that $Q''_1$
is quadratic. Then by putting $Q''_2(x)=0$ in (3.32), it follows
that $Q''_1(x)=0$ for all $x\in X$ and the proof is complete.
\end{proof}

Now we establish the generalized Hyers-Ulam -Rassias stability of
functional equation (1.5) as follows:

\begin{thm}\label{t2} Suppose that a mapping $f:X\rightarrow Y$
satisfies $f(0)=0$ and

$$\|D_f(x,y)\|\leq\phi(x,y)$$ for  all $x,y\in
X.$ If the upper bound $\phi:X\times X\rightarrow [0,\infty)$ is a
mapping such that
$$\sum^{\infty}_{i=0} \{16^i [\frac{1}{3}\phi(\frac{x}{2^i},\frac{x}{2^{i+1}})
+\frac{16}{3}\phi(\frac{x}{2^{i+1}},\frac{x}{2^{i+1}})] +8^i
[\phi(\frac{x}{2^i},\frac{x}{2^i})+4\phi(0,\frac{x}{2^i})]\}<\infty$$
and that $\lim_n 16^n \phi(\frac{x}{2^n},\frac{y}{2^n})=0$ for all
$x,y\in X.$ Then there exist a unique quadratic function
$Q_1:X\rightarrow Y,$ a unique cubic  function $C:X\rightarrow Y$
and a unique quartic  function $Q_2:X\rightarrow Y$ such that

\begin{align*}
\parallel
f(x)-Q_1(x)-C(x)-Q_2(x)\parallel&\leq\frac{1}{12}\sum^{\infty}_{i=0}
(4^i+16^i)[\frac{1}{3}\phi(\frac{x}{2^i},\frac{x}{2^{i+1}})+\frac{16}{3}\phi(\frac{x}{2^{i+1}},\frac{x}{2^{i+1}})]\\
&+\frac{1}{6}\sum^{\infty}_{i=1}8^{i-1}\phi(0,\frac{x}{2^i})+\frac{2}{3}
\sum^{\infty}_{i=1}\phi(\frac{x}{2^i},\frac{x}{2^i}) \hspace
{1.5cm}(3.33)
\end{align*}
for all $x\in X$.
\end{thm}

\begin{proof} Let $f_e(x)=\frac{1}{2}(f(x)+f(-x))$ for all $x\in
X.$ Then $f_e(0)=0,$ $f_e(-x)=f_e(x),$ and
$\|D_{f_e}(x,y)\|\leq\frac{1}{2}[\phi(x,y)+\phi(-x,-y)]$ \hspace
{0.3cm} for all $x,y\in X.$  Hence in view of Theorem 3.4,  there
exist a unique quadratic function $Q_1:X\rightarrow Y$ and a unique
quartic function $Q_2:X\rightarrow Y$ satisfies (3.29). Let
$f_o(x)=\frac{1}{2} (f(x)-f(-x))$. Then $f_o$ is an odd function,
satisfies   $\|D_{f_o}(x,y)\|\leq \frac{1}{2}
[\phi(x,y)+\phi(-x,-y)]$. From Theorem 3.1, it follows that there
exists a unique cubic function $C:X\rightarrow Y$ satisfies (3.2).
 Now it is easy to see that (3.33) holds true for all $x\in X,$
and  the proof of Theorem is complete.
\end{proof}

\begin{cor}\label{t2}
 Let $p>4,$ $\theta\geq0.$ Suppose that a mapping  $f:X\rightarrow Y$ satisfies $f(0)=0,$ and

$$\|D_f (x,y)\|\leq\theta(\|x\|^p+\|y\|^p)$$
for all $x,y\in X.$  Then there exist a unique quadratic function
$Q_1:X\rightarrow Y,$  a unique cubic  function $C:X\rightarrow Y$
and a unique quartic  function $Q_2:X\rightarrow Y$ satisfying

$$\parallel f(x)-Q_1(x)-C(x)-Q_2(x)\parallel\leq
[\frac{33+2^p}{36}(\frac{1}{2^P-4}+\frac{1}{2^P-16})+\frac{3}{2\times(2^P-8)}]
\theta\|x\|^p$$ for all $x\in X.$
\end{cor}

\begin{thm}\label{t2} Suppose that an even function  $f:X\rightarrow
Y$ satisfies

$$\|D_f(x,y)\|\leq\phi(x,y)$$ for all $x,y\in X.$ If the upper bound
 $\phi:X\times X\rightarrow [0,\infty)$ is a mapping such that

$$\sum^{\infty}_{i=1} \frac{1}{4^i} [\phi(2^{i+1}x,2^ix)+
\phi(2^ix,2^ix)]<\infty, \eqno \hspace {1.5cm}(3.34)$$ and that
$\lim_n \frac{1}{4^n}\phi(2^nx,2^ny)=0$ for  all $x,y\in X.$ Then
the limit $$Q_1(x)=\lim_n \frac{1}{4^n}[f(2^{n+1}x)-16f(2^nx)]$$
is a unique quadratic function satisfies (1.5) and

$$\parallel f(2x)-16f(x)-Q_1(x)\parallel\leq \frac{1}{4}\sum^{\infty}_{i=0} \frac{1}{4^i}
[\frac{1}{3}\phi(2^{i+1}x,2^ix)+\frac{16}{3} \phi(2^ix,2^ix)]$$
for all $x\in X.$

\end{thm}

\begin{proof} Similar to the proof of Theorem 3.2,  we can show that $f$ satisfies (3.20). Let

$g(x)=f(2x)-16f(x).$ Then by (3.20)
we have

$$\|\frac{g(2x)}{4}-g(x)\|\leq\frac{1}{4}[\frac{1}{3}\phi(2x,x)+\frac{16}{3}\phi(x,x)].
\eqno \hspace {1.5cm}(3.35)$$ By induction on n and by (3.35) we
have

$$\|\frac{g(2^nx)}{4^n}-g(x)\|\leq\frac{1}{4}\sum^{n-1}_{i=0} \frac{1}{4^i}
[\frac{1}{3}\phi(2^{i+1}x,2^ix)+\frac{16}{3}\phi(2^ix,2^ix)] \eqno
\hspace {1.5cm}(3.36)$$ for all $x\in X.$ Dividing both sides of
(3.36) by $4^m$ and replacing $x$ by $2^mx$ to get the relation

\begin{align*}
\|\frac{g(2^{m+n}x)}{4^{m+n}}-\frac{g(2^mx)}{4^m}\|&\leq\frac{1}{4}\sum^{n-1}_{i=0}
\frac{1}{4^{m+i}}
[\frac{1}{3}\phi(2^{m+i+1}x,2^{m+i}x)+\frac{16}{3}\phi(2^{m+i}x,2^{m+i}x)]\\
&\leq\frac{1}{4}\sum^{m+n-1}_{i=m}\frac{1}{4^i}[\frac{1}{3}\phi(2^{i+1}x,2^ix)+\frac{16}{3}\phi(2^ix,2^ix)].
\end{align*}
By taking $m\rightarrow \infty$ in above inequality and by using
(3.34), we see that the sequence $\{\frac{g(2^nx)}{4^n}\}$ is Cauchy
in Y. Since Y is complete, then $$Q_1(x)=\lim_n
\frac{g(2^nx)}{4^n}=\lim_n \frac{1}{4^n}[f(2^{n+1}x)-16f(2^nx)]$$
exists for all $x\in X.$ The rest of proof is similar to the proof
of Theorem 3.2.

\end{proof}

\begin{thm}\label{t2} Suppose  an even function $f:X\rightarrow Y$
satisfies  $$\|D_f(x,y)\|\leq\phi(x,y)$$ for  all $x,y\in X.$ If the
upper bound $\phi:X\times X\rightarrow [0,\infty)$ is a mapping such
that

$$\sum^{\infty}_{i=1} \frac{1}{16^i} [\phi(2^{i+1}x,2^ix)+\phi(2^ix,2^ix)]
<\infty \eqno \hspace {1.5cm}(3.37)$$ and that $\lim_n
\frac{1}{16^n} \phi(\frac{x}{2^n},\frac{y}{2^n})=0$ for all
$x,y\in X.$ Then the limit $$Q_2 (x)=\lim_n
\frac{1}{16^n}[f(2^{n+1}x)-4f(2^nx)]$$ exists for all $x\in X,$
and $Q_2:X\rightarrow Y$ is a unique quartic function satisfies
(1.5) and

$$\parallel f(2x)-16f(x)-Q_2(x)\parallel\leq \frac{1}{16}\sum^{\infty}_{i=0}
[\frac{1}{3}\phi(2^{i+1}x,2^ix)+\frac{16}{3} \phi(2^ix,2^ix)]$$
for all $x\in X.$
\end{thm}

\begin{proof} The proof is similar to the proof of Theorem 3.3.

\end{proof}

\begin{thm}\label{t2} Suppose that an even function  $f:X\rightarrow Y$ satisfies
$$\|D_f(x,y)\|\leq\phi(x,y)$$ for  all $x,y\in X.$ If the upper
bound $\phi:X\times X\rightarrow [0,\infty)$ satisfying

$$\sum^{\infty}_{i=1} \frac{1}{4^i} [\phi(2^{i+1}x,2^ix)+\phi(2^ix,2^ix)]
<\infty$$ and that $\lim_n \frac{1}{4^n} \phi(2^nx,2^nx)=0$ for
all $x\in X,$ then there exist a unique quadratic function
$Q_1:X\rightarrow Y,$ and a unique quartic  function
$Q_2:X\rightarrow Y$  such that

$$\parallel
f(x)-Q_1(x)-Q_2(x)\parallel\leq\frac{1}{12}\sum^{\infty}_{i=0}
(\frac{1}{4^i}+\frac{1}{16^i})[\frac{1}{3}\phi(2^{i+1}x,2^ix)+\frac{16}{3}\phi(2^ix,2^ix)]$$
for all $x\in X$.
\end{thm}

\begin{proof} The proof is similar to the proof of Theorem 3.4.

\end{proof}

\begin{thm}\label{t2} Suppose that a function  $f:X\rightarrow Y$ satisfies $f(0)=0,$ and

$$\|D_f (x,y)\|\leq\phi(x,y)$$  for  all $x,y\in X.$ If the upper bound $\phi:X\times
X\rightarrow [0,\infty)$ satisfies

$$\sum^{\infty}_{i=1} \frac{1}{4^i} [\phi(2^{i+1}x,2^ix)+\phi(2^ix,2^ix)]
<\infty,$$ and $$\sum^{\infty}_{i=1} \frac{1}{8^i}
[\phi(2^ix,2^ix)+4\phi(0,2^ix)] <\infty$$ for all $x\in X,$ and that
$\lim_n \frac{1}{4^n} \phi(\frac{x}{2^n},\frac{y}{2^n})=0$ for all
$x,y\in X.$ Then there exist a unique quadratic function
$Q_1:X\rightarrow Y,$ a unique cubic function $C:X\rightarrow Y$ and
a unique quartic function $Q_2:X\rightarrow Y$ such that

\begin{align*}
\parallel
f(x)-Q_1(x)-C(x)-Q_2(x)\parallel&\leq\frac{1}{12}\sum^{\infty}_{i=0}
(\frac{1}{4^i}+\frac{1}{16^i})[\frac{1}{3}\phi(2^{i+1}x,2^ix)+\frac{16}{3}\phi(2^ix,2^ix)]\\
&+\frac{1}{6}\sum^{\infty}_{i=0}\frac{1}{8^{i+1}}\phi(0,2^ix)+\frac{2}{3}
\sum^{\infty}_{i=0}\frac{1}{8^{i+1}}\phi(2^ix,2^ix)
\end{align*}
for all $x\in X$.

\end{thm}

\begin{proof} The proof is similar to the proof of Theorem 3.5.

\end{proof}

By Theorem 3.10, we solve the following Hyers-Ulam-Rassias stability
problem for functional equation  (1.5).

\begin{cor}\label{t2}
 Let $p<3,$ and let $\theta$ be a positive real number. Suppose that a mapping  $f:X\rightarrow Y$ satisfies $f(0)=0,$ and

$$\|D_f (x,y)\|\leq\theta(\|x\|^p+\|y\|^p)$$
for all $x,y\in X.$  Then there exist a unique quadratic function
$Q_1:X\rightarrow Y,$  a unique cubic function $C:X\rightarrow Y$
and a unique quartic  function $Q_2:X\rightarrow Y$ satisfying

$$\parallel f(x)-Q_1(x)-C(x)-Q_2(x)\parallel\leq
[(\frac{33+2^p}{9})(\frac{1}{4-2^P}+\frac{4}{16-2^P})
+\frac{3}{2(8-2^P)}]\theta\|x\|^p$$ for all $x\in X$.
\end{cor}

By Corollary 3.11,  we are going to investigate the Hyers-Ulam
stability problem for functional equation (1.5).

\begin{cor}\label{t2}
 Let $\epsilon$ be a positive real number. Suppose that a mapping $f:X\rightarrow Y$
satisfies $f(0)=0$ and $\|D_f(x,y)\|\leq\epsilon$ for all $x,y\in
X.$  Then there exists a unique quadratic  function
$Q_1:X\rightarrow Y,$ a unique cubic function $C:X\rightarrow Y$
and a unique quartic  function $Q_2:X\rightarrow Y$ satisfying

$$\parallel f(x)-Q_1(x)-C(x)-Q_2(x)\parallel \leq \frac{431}{420}\epsilon$$
for all $x\in X$.
\end{cor}

{\small


}
\end{document}